\documentclass[a4paper, 11pt]{article}
\usepackage[utf8]{inputenc}
\usepackage[T1]{fontenc}

\usepackage[
	sorting = nyt,
	backend = biber,
	hyperref = auto,
	backref = true,
	doi = false,
	url = false
]{biblatex}

\usepackage{amsthm}
\usepackage{amssymb}
\usepackage{mathtools}
\usepackage{stmaryrd}	\SetSymbolFont{stmry}{bold}{U}{stmry}{m}{n}	
\usepackage{bm}
\usepackage{nccrules}
\usepackage{euscript}
\usepackage{mathrsfs}
\usepackage{upgreek}
\usepackage{tikz-cd}
\usetikzlibrary{positioning, shapes.geometric, patterns, graphs, decorations.pathmorphing}

\usepackage{paralist}

\usepackage[colorlinks]{hyperref}
\hypersetup{
	linkcolor=blue,
	citecolor=red,
	urlcolor=teal,
	pdftitle = {Decomposed Levi subgroups in BD-covers of classical groups},
	pdfauthor = {Wen-Wei Li}
}

\newcommand{\Z}{\ensuremath{\mathbb{Z}}}
\newcommand{\Q}{\ensuremath{\mathbb{Q}}}

\newcommand{\CC}{\ensuremath{\mathbb{C}}}

\newcommand{\Tr}{\operatorname{tr}}

\newcommand{\Ind}{\operatorname{Ind}}

\newcommand{\Gal}{\operatorname{Gal}}

\newcommand{\identity}{\ensuremath{\mathrm{id}}}

\newcommand{\End}{\operatorname{End}}
\newcommand{\rightiso}{\ensuremath{\stackrel{\sim}{\rightarrow}}}

\newcommand{\Hm}{\operatorname{H}}	
\newcommand{\cate}[1]{\ensuremath{\mathsf{#1}}}	

\newcommand{\Spec}{\operatorname{Spec}}
\newcommand{\Gm}{\ensuremath{\mathbb{G}_\mathrm{m}}}

\newcommand{\Res}{\operatorname{Res}}

\newcommand{\dtimes}[1]{\ensuremath{\underset{#1}{\times}}}

\newcommand{\GL}{\operatorname{GL}}

\newcommand{\SO}{\operatorname{SO}}
\newcommand{\Or}{\operatorname{O}}

\newcommand{\Sp}{\operatorname{Sp}}
\newcommand{\GSp}{\operatorname{GSp}}
\newcommand{\GSpin}{\operatorname{GSpin}}

\newcommand{\pr}{\ensuremath{\mathbf{p}}}
\newcommand{\bmu}{\ensuremath{\bm\mu}}

\newcommand{\rev}{\ensuremath{\mathbf{p}}} 

\theoremstyle{plain}
\newtheorem{proposition}{Proposition}
\newtheorem{lemma}[proposition]{Lemma}
\newtheorem{theorem}[proposition]{Theorem}
\newtheorem{corollary}[proposition]{Corollary}
\theoremstyle{definition}
\newtheorem{definition}[proposition]{Definition}
\newtheorem{definition-theorem}[proposition]{Definition--Theorem}
\newtheorem{definition-proposition}[proposition]{Definition--Proposition}
\newtheorem{remark}[proposition]{Remark}

\numberwithin{equation}{section}
\numberwithin{proposition}{section}

\usepackage{amsmath}
\usepackage[nottoc]{tocbibind}

\usepackage{geometry}
\usepackage{Alegreya}
\usepackage{AlegreyaSans}

\title{Decomposed Levi subgroups in BD-covers of classical groups}
\author{Wen-Wei Li}
\date{}

\DeclareFieldFormat{postnote}{#1}	
\makeatletter
\renewcommand{\l@section}{\@dottedtocline{1}{1.5em}{2.0em}}
\renewcommand{\l@subsection}{\@dottedtocline{2}{4.0em}{3.0em}}
\makeatother

\begin{document}
	
\maketitle

\begin{abstract}
	For finite topological central extensions of $p$-adic classical groups, Heiermann and Wu introduced the notion of decomposed Levi subgroups in their study of intertwining algebras. In this note, we show that for symplectic and special orthogonal groups over local fields, except the split $\SO(4)$, all Levi subgroups are decomposed if the central extension arises from the Brylinski--Deligne construction. Discussions for certain unitary groups, $\SO(4)$ and $\GSpin(2n+1)$ are also given.
\end{abstract}
	
{\scriptsize
	\begin{tabular}{ll}
		\textbf{MSC (2020)} & Primary 22E50; Secondary 11F70 \\
		\textbf{Keywords} & BD-cover, classical groups
	\end{tabular}}
	
\setcounter{tocdepth}{1}
\tableofcontents
	
\section{Introduction}\label{sec:intro}
\subsection{Motivation}
For a classical group $G$ over a non-Archimedean local field $F$, let $[M, \sigma]$ be a Bernstein component of the category $\cate{Rep}_{G(F)}$ of smooth representations over $\CC$ of $G(F)$, the group of $F$-points of $G$. For the corresponding subcategory of $\cate{Rep}_{G(F)}$, Heiermann showed in \cite{Hei11} that the endomorphism algebra of a standard projective generator is the semi-direct product of an affine Hecke algebra of possibly unequal parameters with a finite group algebra. The resulting equivalence of module categories is compatible with parabolic induction, Jacquet modules, and so on. Type theory is not required in this approach.

In \cite{HW25}, Heiermann and Wu extend this important result to $m$-fold covers $\tilde{G}$ of $G(F)$ by assuming that the preimage $\tilde{M} \subset \tilde{G}$ is decomposed, where $m \in \Z_{\geq 1}$. An $m$-fold cover means a central extension of locally compact groups
\[ 1 \to \bmu_m \to \tilde{G} \to G(F) \to 1, \]
where $\bmu_m$ is the group of $m$-th roots of unity in $\CC$. Bernstein's theory generalizes to genuine smooth representations of $\tilde{G}$, i.e.\ the smooth representations of $\tilde{G}$ on which $\bmu_m$ acts tautologically.

Let us clarify the meaning of decomposed Levi subgroups. To begin with, suppose $G$ is not unitary. Levi subgroups $M \subset G$ can be expressed as
\begin{equation}\label{eqn:Levi-intro}
	M = \prod_{i=1}^r \GL_{d_i} \times G^{\flat}
\end{equation}
where $G^{\flat}$ is a smaller classical group of the same type as $G$, possibly trivial. Let $\widetilde{\GL_{d_i}(F)}$ and $\widetilde{G^\flat}$ be the preimages of their $F$-points in $\tilde{M}$, which are again $m$-fold covers. We say $\tilde{M}$ is decomposed if
\begin{equation}\label{eqn:Levi-intro-cover}
	\tilde{M} \simeq \left( \prod_{i=1}^r \widetilde{\GL_{d_i}} \times \widetilde{G^\flat} \right) \bigg/ \left\{ (z_1, \ldots, z_r, z^\flat) \in \bmu_m^{r+1}, \; z_1 \cdots z_r z^\flat = 1 \right\}
\end{equation}
as topological central extensions of $M(F)$ by $\bmu_m$. When $G$ is unitary with respect to a quadratic extension $E|F$, one replaces $\GL_{d_i}$ by the restriction of scalars $\Res_{E|F} \GL_{d_i}$.

Decomposability plays a crucial role in \cite{HW25}; for instance, it is satisfied by all Levi subgroups of the metaplectic twofold extension of $\Sp(2n, F)$. On the other hand, decomposability fails for the Kazhdan--Patterson covers \cite{KP84} of $\GL_n(F)$ in general; this failure leads to the theory of metaplectic tensor products.

In view of \cite{HW25}, we are thus led to the following natural question, which makes sense for all local fields $F$ with $\mathrm{char}(F) \neq 2$ and all classical groups $G$ over $F$.

\textbf{Question}. Given $\tilde{G}$ and a Levi subgroup $M \subset G$, how to determine whether $\tilde{M}$ is decomposed or not?

In this note we focus on the $m$-fold covers constructed by Brylinski and Deligne \cite{BD01}, abbreviated as $m$-fold BD-covers, where $m \mid m_F$. Here $m_F$ is the number of roots of unity in $F$ (resp.\ $m_F := 1$) when $F \neq \CC$ (resp.\ when $F = \CC$). After fixing an embedding $\epsilon: \mu_m(F) \hookrightarrow \CC^{\times}$, the BD-covers arise from central extensions
\[ 0 \to \mathbf{K}_2 \to \tilde{\mathcal{G}} \to G \to 1 \]
of sheaves on the big Zariski site of $\Spec(F)$, where $\mathbf{K}_2$ is the sheaf associated with Quillen's K-theory functors. The BD-covers comprise most known covers of arithmetic significance, including the metaplectic group. Gan--Gao--Weissman \cite{GGW} proposed a Langlands program for genuine representations of BD-covers.

\subsection{Main results and strategy}
The main results are summarized below. In what follows, $\tilde{G}$ is assumed to be a BD-cover of $G(F)$.
\begin{itemize}
	\item Suppose that $G$ is the classical group attached to a pair $(W, h)$, where $W$ is a finite-dimensional $F$-vector space and $h: W \times W \to F$ is a non-degenerate bilinear form, in one of the following cases:
	\begin{itemize}
		\item $(W, h)$ is symplectic,
		\item $(W, h)$ is quadratic and $\dim W$ is odd,
		\item $(W, h)$ is quadratic and $\dim W \geq 6$ is even.
	\end{itemize}
	Then $\tilde{M}$ is decomposed for all Levi subgroups $M \subset G$ (Corollary \ref{prop:main-topological}).
	\item If $G$ is a quasisplit unitary group attached to an Hermitian space $W$ with $\dim W \geq 3$, or the split $\SO(4)$, there are examples, over suitable $F$, of BD-covers $\tilde{G}$ in which $\tilde{T}$ is not decomposed (Propositions \ref{prop:unitary-example-F}, \ref{prop:odd-unitary-example-F} and Remark \ref{rem:SO4-example-F}). Here $T$ is a minimal Levi subgroup, which is also a maximal $F$-torus in $G$.
	\item If $G$ is a non-split $\SO(4)$, all Levi subgroups are decomposed (Proposition \ref{prop:SO4-example-nonsplit}).
	\item Discussions for full orthogonal groups (Proposition \ref{prop:O}) and $\GSpin(2n+1)$ (Proposition \ref{prop:GSpin}) are also given.
\end{itemize}

The idea goes as follows. The right-hand side of \eqref{eqn:Levi-intro-cover} can be seen as an external Baer sum of covers. To say $\tilde{M}$ is decomposed amounts to saying that $\tilde{M}$ is isomorphic to the external Baer sum of the preimages of $F$-points of factors in \eqref{eqn:Levi-intro}, say through the multiplication map in $\tilde{G}$. For BD-covers, we try to deduce this isomorphism from the $\mathbf{K}_2$-level, namely for the pull-back $\tilde{\mathcal{M}}$ of $\tilde{\mathcal{G}}$ to $M$. In both settings, decomposability is equivalent to the condition that different factors in \eqref{eqn:Levi-intro} commute in the central extension.

Consider any connected reductive $F$-group $G$ with a chosen maximal $F$-torus $T$. Brylinski and Deligne classified the central $\mathbf{K}_2$-extensions of $G$ in terms of data $(Q, \mathcal{D}, \varphi)$ where $Q: Y := \mathrm{X}_*(T_{\overline{F}}) \to \Z$ is a Weyl- and Galois-invariant quadratic form. Now take $T \subset M \subset G$ in our original setting and consider the decomposition \eqref{eqn:Levi-intro}. To decompose $\tilde{\mathcal{M}}$, we are reduced to showing that the corresponding summands in $Y$ are mutually orthogonal relative to $B_Q(y, y') := Q(y+y') - Q(y) - Q(y')$ (Proposition \ref{prop:Baer-orthogonal}), which is elementary. This leads to the $\mathbf{K}_2$-level decomposition (Theorem \ref{prop:main}).

To produce examples of non-decomposed Levi subgroups in $\tilde{G}$ requires more work: one has to express the commutator pairing in BD-covers in terms of Hilbert symbols, or Tate's local duality in the unitary case. Note that, for BD-covers of unitary groups arising from Savin covers of general linear groups (instead of Kazhdan--Patterson covers), the Levi subgroups always decompose; see Remark \ref{rem:Savin-cover}.

All the required machinery is available from the literature. In particular, we rely on the crucial fact that the Brylinski--Deligne equivalence preserves Baer sums.

On the other hand, since the theory of Brylinski--Deligne \cite{BD01} only concerns connected reductive groups, the results for full orthogonal groups given in \S\ref{sec:full-orthogonal} will be very limited.

\subsection{Layout}
In \S\ref{sec:cext} we recollect the required terminology for topological central covers, Baer sums and external Baer sums. The Brylinski--Deligne formalism of central $\mathbf{K}_2$-extensions is reviewed in \S\ref{sec:BD}, together with certain auxiliary results proved. In \S\ref{sec:classical-group} we turn to classical groups and prove the main theorems. Counterexamples for unitary groups and $\SO(4)$ are given in \S\ref{sec:counterexample-unitary} and \S\ref{sec:counterexample-SO4}, respectively. Finally, \S\ref{sec:full-orthogonal} and \S\ref{sec:GSpin-odd} discuss the cases of full orthogonal groups and $\GSpin(2n+1)$, respectively.

\subsection{Acknowledgement}
The author is grateful to Volker Heiermann and Chenyan Wu for kindly sharing their work \cite{HW25} and highlighting the issue of decomposability, as well as to Fan Gao for assuring that Levi subgroups ought to decompose in BD-covers of many classical groups.

The author also thanks the anonymous referee for the careful reading and suggestions.

This work is supported by NSFC (Grants 11922101 and 12321001).

\subsection{AI use declaration}
Large language models (Kimi K3.1 Pro and GLM-5.4) were used in revising the manuscript and checking the mathematical soundness of the revision. The author bears full responsibility for the correctness of the results and proofs, and for the integrity and accuracy of the references.

\subsection{Notation and conventions}
For a local field $F$, we let $m_F \in \Z_{\geq 1}$ be the number of roots of unity in $F$, with the convention that $m_{\CC} := 1$. For every $m \mid m_F$ we denote by $(\cdot, \cdot)_{F, m}$ the Hilbert symbol of degree $m$ for $F$. When $F$ is non-Archimedean, the ring of integers (resp.\ its maximal ideal) is written as $\mathfrak{o}_F$ (resp.\ $\mathfrak{p}_F$).

Field extensions are written as $E|F$, and for Galois extensions we denote by $\Gal(E|F)$ its Galois group. The absolute Galois group of $F$ is denoted by $\Gal_F$.

For a field $F$ with $\mathrm{char}(F) \neq 2$, a quadratic (resp.\ symplectic) $F$-vector space means a pair $(V, q)$ where $V$ is a finite-dimensional $F$-vector space and $q: V \times V \to F$ is bilinear, symmetric (resp.\ alternating) and non-degenerate; the shorthand $V$ will also be used. The corresponding orthogonal and special orthogonal groups (resp.\ symplectic groups) are denoted by $\Or(V, q)$ and $\SO(V, q)$ (resp.\ $\Sp(V, q)$), or simply as $\Or(V)$ and $\SO(V)$ (resp.\ $\Sp(V)$), respectively.

Given a quadratic extension $E|F$ with $\mathrm{char}(F) \neq 2$, Hermitian spaces $(W, h)$ (shorthand: $W$) and unitary groups are defined similarly, where $h: W \times W \to E$ is the Hermitian form.

For a scheme $X$ over $\Spec(F)$, we denote by $X(F)$ the set of its $F$-points. For an extension $F' | F$ of fields, we denote by $X_{F'}$ the base-change of $X$ to $F'$. When $[F':F]$ is finite, the functor  $\Res_{F'|F}$ of Weil restriction sends affine $F'$-groups to affine $F$-groups.

When $G$ is a connected reductive $F$-group with a split maximal $F$-torus $T$, we denote by $\mathrm{W}(G, T)$ the Weyl group. Denote by $G_{\mathrm{sc}}$ the simply connected covering of $[G, G]$. Lie algebras are denoted in $\mathfrak{Fraktur}$ letters. The cocharacter (resp.\ character) lattice of $T$ is denoted by $\mathrm{X}_*(T)$ (resp.\ $\mathrm{X}^*(T)$).

For every $m \in \Z_{\geq 1}$, we let $\mu_m$ be the group scheme of $m$-th roots of unity. Put $\bmu_m := \mu_m(\CC)$.

Variables in a polynomial ring are denoted by $\mathsf{X}_1, \mathsf{X}_2, \ldots$.

\section{Topological central extensions}\label{sec:cext}
Given a finite abelian group $A$ endowed with the discrete topology, and a locally compact group $H$, we let $\cate{CExt}(H, A)$ be the category of topological central extensions
\[ 1 \to A \to \mathcal{E} \xrightarrow{\rev} H \to 1 \]
as explained in \cite[1.1 Definition]{GGW}. Morphisms in $\cate{CExt}(H, A)$ are commutative diagrams of the form
\[\begin{tikzcd}
	1 \arrow[r] & A \arrow[equal, d] \arrow[r] & \mathcal{E} \arrow[d, "\theta"] \arrow[r] & H \arrow[equal, d] \arrow[r] & 1 \\
	1 \arrow[r] & A \arrow[r] & \mathcal{E}' \arrow[r] & H \arrow[r] & 1  
\end{tikzcd}\]
where $\theta$ is a continuous group homomorphism. By the discussions in \cite[1.2 Definition]{GGW}, $\cate{CExt}(H, A)$ is a groupoid.

Every continuous homomorphism $f: H' \to H$ (resp.\ $\alpha: A \to A'$) induces a functor $f^*: \cate{CExt}(H, A) \to \cate{CExt}(H', A)$ (resp.\ $\alpha_*: \cate{CExt}(H, A) \to \cate{CExt}(H, A')$): it is the pull-back (resp.\ push-out) of central extensions along $f$ (resp.\ $\alpha$).

\begin{definition}[see {\cite[p.2]{GGW}}]
	Given $\mathcal{E}_1, \mathcal{E}_2$ in $\cate{CExt}(H, A)$, define $\mathcal{E}_1 \dot{+} \mathcal{E}_2$ to be the push-out of the object
	\[ \left\{ (h_1, h_2) \in \mathcal{E}_1 \times \mathcal{E}_2 : \rev_1(h_1) = \rev_2(h_2) \right\} \]
	of $\cate{CExt}(H, A \times A)$ along $\mathrm{mult}: A \times A \to A$. This yields a functor $\dot{+}: \cate{CExt}(H, A) \times \cate{CExt}(H, A) \to \cate{CExt}(H, A)$, called the Baer sum.
\end{definition}

The operation $\dot{+}$ makes $\cate{CExt}(H, A)$ into a commutative Picard category, whose zero object is given by the split extension $A \times H$. Furthermore, $\dot{+}$ commutes with pull-backs up to a canonical isomorphism.

Given objects $\mathcal{E}_1, \ldots, \mathcal{E}_k$ of $\cate{CExt}(H, A)$, the Baer sum $\mathcal{E}_1 \dot{+} \cdots \dot{+} \mathcal{E}_k$ is defined similarly.

Here comes a variant. Given locally compact groups $H_1, \ldots, H_k$, we denote by $p_i: H_1 \times \cdots \times H_k \to H_i$ the projection map, for all $1 \leq i \leq k$.

\begin{definition}
	Given objects $\mathcal{E}_i$ of $\cate{CExt}(H_i, A)$ for $i = 1, \ldots, k$, put
	\begin{equation*}
		\mathcal{E}_1 \dot{\boxplus} \cdots \dot{\boxplus} \mathcal{E}_k := p_1^* \mathcal{E}_1 \dot{+} \cdots \dot{+} p_k^*\mathcal{E}_k,
	\end{equation*}
	called the external Baer sum of $\mathcal{E}_1, \ldots, \mathcal{E}_k$.
\end{definition}

The definition above is given in terms of ``internal'' Baer sum in $\cate{CExt}(\prod_i H_i, A)$. Alternatively, it can be described directly as the push-out of the central extension
\[ 1 \to A^k \to \prod_{i=1}^k \mathcal{E}_i \to \prod_{i=1}^k H_i \to 1 \]
along $\mathrm{mult}: A^k \to A$, and then $\dot{+}$ on $\cate{CExt}(H, A)$ can be redefined as the pull-back of $\dot{\boxplus}$ (in the case $H_1 = \cdots = H_k = H$) along the diagonal map $H \to H^k$.

\begin{definition}\label{def:comm}
	Given $\mathcal{E}$ in $\cate{CExt}(H, A)$, let $h_1, h_2 \in H$ be commuting elements. Take $\tilde{h}_1 \in \pr^{-1}(h_1)$ and $\tilde{h}_2 \in \pr^{-1}(h_2)$. Then the commutator
	\[ \mathrm{comm}(h_1, h_2) := \tilde{h}_1 \tilde{h}_2 \tilde{h}_1^{-1} \tilde{h}_2^{-1} \]
	lies in $A$, and is independent of the choice of preimages.
\end{definition}

In what follows, we write $\mathbf{1}$ for the constant function $1$ valued in a given group.

\begin{proposition}\label{prop:recognition-Baer}
	Given locally compact groups $H_1, \ldots, H_k$ and $\mathcal{E} \in \cate{CExt}(\prod_i H_i, A)$, let $\mathcal{E}_i$ be the pull-back of $\mathcal{E}$ along $H_i \hookrightarrow \prod_j H_j$. Then $\mathcal{E}$ is isomorphic to $\mathcal{E}_1 \dot{\boxplus} \cdots \dot{\boxplus} \mathcal{E}_k$ if and only if
	\[ \mathrm{comm}|_{H_i \times H_j} = \mathbf{1} \quad \text{for all} \; 1 \leq i \neq j \leq k, \]
	in which case there is a canonical isomorphism
	\[ \mathcal{E}_1 \dot{\boxplus} \cdots \dot{\boxplus} \mathcal{E}_k \rightiso \mathcal{E}. \]
\end{proposition}
\begin{proof}
	To begin with, assume $k = 2$. The ``only if'' direction is trivial: $\mathrm{comm}|_{H_1 \times H_2} = \mathbf{1}$ holds for $\mathcal{E}_1 \dot{\boxplus} \mathcal{E}_2$ by the direct construction of $\dot{\boxplus}$.
	
	For the ``if'' direction, the multiplication map $\mathcal{E}_1 \times \mathcal{E}_2 \to \mathcal{E}$ within $\mathcal{E}$ is a continuous homomorphism of groups since $\mathrm{comm}|_{H_1 \times H_2} = \mathbf{1}$. It is trivial on the anti-diagonal of $A \times A$. Passing to quotients, we obtain a canonical morphism
	\[ \mathcal{E}_1 \dot{\boxplus} \mathcal{E}_2 \to \mathcal{E} \]
	in the groupoid $\cate{CExt}(H_1 \times H_2, A)$.
	
	The case for general $k$ is analogous; it can also be proved inductively from the $k=2$ case.
\end{proof}

\begin{remark}\label{rem:recognition-Baer}
	There is also a version of Proposition \ref{prop:recognition-Baer} for central extensions of sheaves of groups. The same formal proof applies, with commutators interpreted section-wise.
\end{remark}

\section{Brylinski--Deligne theory and Baer sums}\label{sec:BD}
Let $F$ be a field and $G$ be a connected reductive $F$-group. In \cite{BD01}, Brylinski and Deligne considered central extensions
\begin{equation}\label{eqn:K2-CExt}
	0 \to \mathbf{K}_2 \to \tilde{\mathcal{G}} \to G \to 1
\end{equation}
of sheaves on the big Zariski site $(\Spec F)_{\mathrm{Zar}}$, where $\mathbf{K}_j$ stands for the Zariski sheaf of abelian groups associated with the presheaf $U \mapsto K_j(U)$ of Quillen's higher $K$-theory for schemes, for all $j \in \Z_{\geq 0}$.

Central extensions can be studied in any topos; we refer to \cite[p.3769]{Weis16} for a summary. In particular, as in the topological setting of \S\ref{sec:cext}, such central extensions form a category $\cate{CExt}(G, \mathbf{K}_2)$, and it has the following features.
\begin{itemize}
	\item Every homomorphism $f: G' \to G$ induces the pull-back functor
	\[ f^*: \cate{CExt}(G, \mathbf{K}_2) \to \cate{CExt}(G', \mathbf{K}_2), \quad (f_1 f_2)^* \simeq f_2^* f_1^* . \]
	\item The operation of Baer sums $\dot{+}$ makes $\cate{CExt}(G, \mathbf{K}_2)$ into a commutative Picard category, with zero object being the split extension. There is also a notion of external Baer sums $\dot{\boxplus}$.
	\item Objects of $\cate{CExt}(G, \mathbf{K}_2)$ can also be viewed as multiplicative $\mathbf{K}_2$-torsors over $G$, see \cite[\S 1]{BD01}. In this way, we see that $\cate{CExt}(G, \mathbf{K}_2)$ is a groupoid: every morphism therein is a morphism between $\mathbf{K}_2$-torsors, hence invertible.
	\item The $\mathbf{K}_2$-valued commutator map $\mathrm{comm}$ is defined on the scheme $\{ (g_1, g_2) \in G^2 : g_1 g_2 = g_2 g_1 \}$, cf.\ Definition \ref{def:comm}.
\end{itemize}

In fact, \cite{BD01} allows more general bases than $\Spec F$, and $\cate{CExt}(G, \mathbf{K}_2)$ has functoriality for base change.

Taking $F$-points in \eqref{eqn:K2-CExt} yields a central extension of groups
\begin{equation}\label{eqn:K2-CExt-F}
	0 \to K_2(F) \to \tilde{\mathcal{G}}(F) \to G(F) \to 1.
\end{equation}

Now assume $F$ is a local field; recall that $m_F \in \Z_{\geq 1}$ is the number of roots of unity in $F$ if $F \neq \CC$, and $m_{\CC} := 1$. Assume $m \mid m_F$. Pushing-out \eqref{eqn:K2-CExt-F} along the Hilbert symbol $(\cdot, \cdot)_{F, m}$, we obtain a central extension of $G(F)$ by $\mu_m(F)$, denoted by
\[ 1 \to \mu_m(F) \to \tilde{G} \to G(F) \to 1. \]
According to \cite[\S 10]{BD01}, the above is naturally a topological central extension. Upon fixing an embedding $\epsilon: \mu_m(F) \hookrightarrow \CC^{\times}$, one may identify $\tilde{G}$ with an object of $\cate{CExt}(G(F), \bmu_m)$.

\begin{proposition}\label{prop:F-pts-Baer}
	The functor $\cate{CExt}(G, \mathbf{K}_2) \to \cate{CExt}(G(F), \bmu_m)$ so obtained respects pull-backs and Baer sums of central extensions, i.e.\ it is a braided monoidal or ``additive'' functor between commutative Picard categories.
	
	Furthermore, the functor preserves external Baer sums.
\end{proposition}
\begin{proof}
	Consider a homomorphism $f: G' \to G$. The pull-back of $\tilde{\mathcal{G}}$ along $f$ is $\tilde{\mathcal{G}}' = \tilde{\mathcal{G}} \dtimes{G} G'$ and we have $\tilde{\mathcal{G}}'(F) = \tilde{\mathcal{G}}(F) \dtimes{G(F)} G'(F)$, in which $K_2(F) = \mathbf{K}_2(\Spec F)$ embeds as $a \mapsto (a, 1)$. Pushing-out along $K_2(F) \to \mu_m(F)$ only affects the first component, thus gives rise to $\tilde{G} \dtimes{G(F)} G'(F)$, the pull-back of $\tilde{G}$ along $G'(F) \to G(F)$.
	
	For Baer sums, it suffices to consider $\tilde{\mathcal{G}}_1 \dot{+} \tilde{\mathcal{G}}_2$. Its formation involves: (a) some projective limits of Zariski sheaves, and (b) pushing-out along the addition map $\mathbf{K}_2 \times \mathbf{K}_2 \to \mathbf{K}_2$. As (a) commutes with taking $F$-points, it remains to show that (b) has the same property. But this is immediate since $\Spec F$ is just a point for the Zariski topology.
	
	Finally, external Baer sums are described in terms of $\dot{+}$ and certain pull-backs of central extensions, hence are preserved as well.
\end{proof}

\begin{definition}
	Objects of $\cate{CExt}(G(F), \bmu_m)$ arising from $\cate{CExt}(G, \mathbf{K}_2)$ are called $m$-fold BD-covers.
\end{definition}

Return to the case of a general field $F$ with a chosen separable closure $\overline{F}$. Let $G$ be a connected reductive $F$-group. Fix a maximal $F$-torus $T$ of $G$, with preimage $T_{\mathrm{sc}}$ in $G_{\mathrm{sc}}$, and let $Y$ (resp.\ $Y_{\mathrm{sc}}$) be $\mathrm{X}_*(T_{\overline{F}})$ (resp.\ $\mathrm{X}_*(T_{\mathrm{sc}, \overline{F}})$). In order to understand $\cate{CExt}(G, \mathbf{K}_2)$, we consider triples $(Q, \mathcal{D}, \varphi)$ where:
\begin{itemize}
	\item $Q: Y \to \Z$ is an integer-valued quadratic form, so that
	\begin{equation}\label{eqn:BQ}
		B_Q(y_1, y_2) := Q(y_1 + y_2) - Q(y_1) - Q(y_2)
	\end{equation}
	is a bilinear form $Y \times Y \to \Z$, which is required to be $\mathrm{W}(G_{\overline{F}}, T_{\overline{F}}) \rtimes \Gal_F$-invariant;
	\item $1 \to \Gm \to \mathcal{D} \to Y \to 0$ is a central extension of sheaves of groups on the site $(\Spec F)_{\text{ét}}$ satisfying
	\begin{equation}\label{eqn:comm-D}
		(-1)^{B_Q(y_1, y_2)} = \mathrm{comm}(y_1, y_2)
	\end{equation}
	for all $y_1, y_2 \in Y$;
	\item by \cite[Proposition 3.11 and Compatibility 4.9]{BD01}, there is a natural central extension
	\[ 1 \to \Gm \to \mathcal{D}_{\mathrm{sc}} \to Y_{\mathrm{sc}} \to 0 \]
	attached to $Q|_{Y_{\mathrm{sc}}}$, and $\varphi: \mathcal{D}_{\mathrm{sc}} \to \mathcal{D}$ is a morphism of sheaves of groups on $(\Spec F)_{\text{ét}}$, such that the diagram
	\[\begin{tikzcd}
		1 \arrow[r] & \Gm \arrow[equal, d] \arrow[r] & \mathcal{D}_{\mathrm{sc}} \arrow[d, "\varphi"] \arrow[r] & Y_{\mathrm{sc}} \arrow[hookrightarrow, d] \arrow[r] & 0 \\
		1 \arrow[r] & \Gm \arrow[r] & \mathcal{D} \arrow[r] & Y \arrow[r] & 0
	\end{tikzcd}\]
	commutes.
\end{itemize}

The data above satisfy étale descent by definition. They form a category $\cate{BD}(G, T)$: morphisms from $(Q, \mathcal{D}, \varphi)$ to $(Q', \mathcal{D}', \varphi')$ exist only when $Q = Q'$, in which case they are defined as morphisms of central extensions $\mathcal{D} \to \mathcal{D}'$ making the obvious diagram commute.

Baer sums of triples $(Q_i, \mathcal{D}_i, \varphi_i)$ (for $i = 1, \ldots, k$) can be defined: simply take $Q := \sum_i Q_i$, $\mathcal{D} := \dot{+}_i \mathcal{D}_i$ (Baer sum of central extensions of sheaves), and $\varphi: \mathcal{D}_{\mathrm{sc}} \to \mathcal{D}$ is induced by the functoriality of Baer sums (one can show that $\mathcal{D}_{\mathrm{sc}} = \dot{+}_i \mathcal{D}_{i, \mathrm{sc}}$ since $Q = \sum_i Q_i$). This makes $\cate{BD}(G, T)$ into a commutative Picard category. See \cite[\S 2.1]{Weis16}.

Using Weyl-invariance, it is easily seen that different choices of $T$ give rise to equivalent $\cate{BD}(G, T)$.

\begin{theorem}\label{prop:BD}
	There is a natural equivalence of commutative Picard categories $\cate{CExt}(G, \mathbf{K}_2) \rightiso \cate{BD}(G, T)$.
	
	Furthermore, suppose $f: G' \to G$ is a homomorphism between connected reductive $F$-groups and $T' \subset G'$ is a maximal $F$-torus such that $f(T') \subset T$ and $f$ induces $\mathrm{W}(G'_{\overline{F}}, T'_{\overline{F}}) \to \mathrm{W}(G_{\overline{F}}, T_{\overline{F}})$, then the diagram
	\[\begin{tikzcd}
		\cate{CExt}(G, \mathbf{K}_2) \arrow[r, "\sim"] \arrow[d] & \cate{BD}(G, T) \arrow[d] \\
		\cate{CExt}(G', \mathbf{K}_2) \arrow[r, "\sim"'] & \cate{BD}(G', T')
	\end{tikzcd}\]
	commutes up to a natural isomorphism, where the vertical functors are pull-backs. 
\end{theorem}
\begin{proof}
	The first part is the main result of Brylinski--Deligne \cite[Theorem 7.2]{BD01}; the assertion about Baer sums is in \cite[Theorem 2.2]{Weis16}.
	
	The second part can be found in \cite[Proposition 4.3]{Cai23}.
\end{proof}

Until the end of this section, we consider connected reductive $F$-groups $G_1, \ldots, G_k$. Put $G := \prod_i G_i$ with projection maps $p_i: G \to G_i$ for $i = 1, \ldots, k$. Fix maximal $F$-tori $T_i \subset G_i$ and let $Y_i = \mathrm{X}_*(T_{i, \overline{F}})$ for each $i$; put $T := \prod_i T_i$ and $Y := \bigoplus_i Y_i$.

Define the external Baer sum of objects $(Q_i, \mathcal{D}_i, \varphi_i) \in \cate{BD}(G_i, T_i)$ as the Baer sum of their pull-backs along $p_i$ (over $i = 1, \ldots, k$) in $\cate{BD}(G, T)$, again denoted by $\dot{\boxplus}$. In particular, the datum $Q: Y \to \Z$ is the orthogonal sum of $Q_1, \ldots, Q_k$.

\begin{corollary}\label{prop:BD-external}
	The equivalence in Theorem \ref{prop:BD} preserves external Baer sums.
\end{corollary}
\begin{proof}
	Suppose that $\tilde{\mathcal{G}}_i \in \cate{CExt}(G_i, \mathbf{K}_2)$ corresponds to $(Q_i, \mathcal{D}_i, \varphi_i)$ for each $1 \leq i \leq k$. Recall that
	\[ \tilde{\mathcal{G}} := \tilde{\mathcal{G}}_1 \dot{\boxplus} \cdots \dot{\boxplus} \tilde{\mathcal{G}}_k \simeq p_1^* \tilde{\mathcal{G}}_1 \dot{+} \cdots \dot{+} p_k^* \tilde{\mathcal{G}}_k. \]
	
	The equivalence in the first part of Theorem \ref{prop:BD} preserves $\dot{+}$; it is also compatible with $p_i^*$ by the second part. Hence $\tilde{\mathcal{G}}$ corresponds to the external Baer sum of $(Q_i, \mathcal{D}_i, \varphi_i)_{i=1}^k$. 
\end{proof}

In the next result, we still assume $G = \prod_{i=1}^k G_i$, embed $G_i$ (resp.\ $T_i$) into $G$ (resp.\ $T$) for each $i$. Given $(Q, \mathcal{D}, \varphi) \in \cate{BD}(G, T)$, we denote its pull-back to $\cate{BD}(G_i, T_i)$ by $(Q_i, \mathcal{D}_i, \varphi_i)$, for each $i$.

\begin{proposition}\label{prop:Baer-orthogonal}
	An object $(Q, \mathcal{D}, \varphi)$ of $\cate{BD}(G, T)$ is isomorphic to the external Baer sum of $(Q_i, \mathcal{D}_i, \varphi_i)$, for $i = 1, \ldots, k$, if and only if $Y_1, \ldots, Y_k$ are mutually orthogonal relative to $B_Q$.
\end{proposition}
\begin{proof}
	What needs proof is the ``if'' part. First off, suppose $G = G_{\mathrm{sc}}$, then $(Q, \mathcal{D}, \varphi)$ (resp.\ $(Q_i, \mathcal{D}_i, \varphi_i)$) amounts to the single datum $Q$ (resp.\ $Q_i$). As $Q$ is the orthogonal sum of $Q_1, \ldots, Q_k$, it follows that $(Q, \mathcal{D}, \varphi)$ is the desired external Baer sum.
	
	For general $G$, note that $\mathcal{D}_1, \ldots, \mathcal{D}_k$ commute with each other inside $\mathcal{D}$ by \eqref{eqn:comm-D}, hence their external Baer sum is $\mathcal{D}$ (see Remark \ref{rem:recognition-Baer}). Note that $\mathcal{D}_{\mathrm{sc}}$ is also the external Baer sum of various $\mathcal{D}_{i, \mathrm{sc}}$ by the previous paragraph. Also, $\varphi$ restricts to $\varphi_i: \mathcal{D}_{i, \mathrm{sc}} \to \mathcal{D}_i$ for each $1 \leq i \leq k$. Summing up, we obtain a morphism from $\dot{\boxplus}_{i=1}^k (Q_i, \mathcal{D}_i, \varphi_i)$ to $(Q, \mathcal{D}, \varphi)$.
\end{proof}

\section{Classical groups and their covers}\label{sec:classical-group}
Consider a field $F$ with $\mathrm{char}(F) \neq 2$, together with an extension of fields $E | F$ which is either quadratic or satisfies $E = F$. We consider the data $(W, h)$ where
\begin{itemize}
	\item $W$ together with $h: W \times W \to E$ is either an Hermitian (if $[E:F] = 2$), symplectic or quadratic vector space (if $E = F$);
	\item $G := G(W, h)$ (or simply $G(W)$) is the unitary, symplectic, or special orthogonal group attached to $(W, h)$.
\end{itemize}
These are the classical groups over $F$ under consideration in this article. The general linear groups are excluded here.

The Levi subgroups of $G$ take the form
\begin{equation}\label{eqn:Levi}
	M = \prod_{i=1}^r \Res_{E|F} \GL_{d_i} \times G^\flat, \quad G^\flat := G(W^\flat),
\end{equation}
which is obtained from an orthogonal decomposition
\begin{equation}
	W = \bigoplus_{i=1}^r (\ell_i \oplus \ell_i^*) \oplus W^\flat, \quad r \in \Z_{\geq 0},
\end{equation}
where $\ell_i \oplus \ell_i^*$ and $W^\flat$ are non-degenerate $E$-subspaces of $W$, and both $\ell_i$ and $\ell_i^*$ are totally isotropic with $d_i := \dim_E \ell_i = \dim_E \ell^*_i \in \Z_{\geq 1}$, for all $1 \leq i \leq r$.

Fix a separable closure $\overline{F}$ of $F$. In order to understand pull-backs to $M$ of central $\mathbf{K}_2$-extensions of $G$, we use the following tool.

\begin{lemma}\label{prop:main-aux}
	Assume $F = \overline{F}$. Let $G = G(W, h)$ be a classical group over $F$ in one of the following cases:
	\begin{itemize}
		\item $(W, h)$ is symplectic, $\dim W > 0$,
		\item $(W, h)$ is quadratic and $\dim W \geq 3$, but $\dim W \neq 4$.
	\end{itemize}
	Assume furthermore that
	\begin{itemize}
		\item an orthogonal decomposition $W = W_1 \oplus \cdots \oplus W_s$ into non-degenerate subspaces is given, with $h_i := h|_{W_i \times W_i}$;
		\item $T = \prod_{i=1}^s T_i \subset \prod_{i=1}^s G(W_i, h_i) \subset G$ is a maximal torus of $G$; note that the maximality of the torus $T$ implies that at most one $W_i$ is of odd dimension.
	\end{itemize}

	Write $Y = \bigoplus_{i=1}^s Y_i$, where $Y = \mathrm{X}_*(T)$ and $Y_i = \mathrm{X}_*(T_i)$ for all $1 \leq i \leq s$. Let $Q: Y \to \Z$ be a $\mathrm{W}(G, T)$-invariant quadratic form and define $B_Q$ by \eqref{eqn:BQ}. Then $Y_i$ is orthogonal to $Y_j$ relative to $B_Q$ for all $1 \leq i \neq j \leq s$.
\end{lemma}
\begin{proof}
	Note that we are free to replace each $T_i$ by its conjugates in $G(W_i, h_i)$.
	
	The problem is insensitive to $F$ in the following sense. Up to isomorphisms, each $(W_i, h_i)$ can be defined over $\Z[\frac{1}{2}]$, so that $T_i \subset G(W_i, h_i)$ and $T \subset G$ have split smooth models over $\Z[\frac{1}{2}]$, for suitable choices of $T_i$. Going through $F \leftarrow \Z[\frac{1}{2}] \to \overline{\Q}$, we infer that it suffices to deal with the special case $F = \overline{\Q}$.
	
	Assume henceforth $F = \overline{\Q}$. By the assumptions, the Coxeter data underlying $\mathrm{W}(G, T)$ is irreducible, thus up to passage from $\Z$-valued to $F$-valued quadratic forms, the invariant quadratic forms on $Y$ are proportional to each other by \cite[Ch.~V, \S 4.7]{Bou02}. It suffices to exhibit a nonzero invariant quadratic form $Q': Y \otimes F \to F$, such that $Y_i$ is orthogonal to $Y_j$ relative to $Q'$ for each $1 \leq i \neq j \leq s$.
	
	Denote by $B_{\mathrm{K}}$ the Killing form for $\mathfrak{g}$, and define the trace form
	\[ B_{\mathrm{tr}}(y', y'') := \Tr(y' y'' | W), \quad y', y'' \in \mathfrak{g} \subset \End(W). \]
	
	Clearly, $B_{\mathrm{tr}}$ is $G$-invariant, symmetric and nonzero. From the simplicity of $\mathfrak{g}$ we deduce that
	\begin{equation}\label{eqn:Killing}
		\exists c \in F^{\times}, \quad B_{\mathrm{tr}} = c B_{\mathrm{K}}.
	\end{equation}
	
	Embed $Y$ into $Y \otimes_{\Z} F \simeq \mathfrak{t}$. Similarly, $Y_i \hookrightarrow \mathfrak{t}_i$ for all $i$. The restriction of $B_{\mathrm{K}}$ to $\mathfrak{t} \times \mathfrak{t}$ is known to be $\mathrm{W}(G, T)$-invariant and non-degenerate, hence so is the restriction of $B_{\mathrm{tr}}$ by \eqref{eqn:Killing}. Take $Q': Y \otimes F \to F$ to be the invariant quadratic form corresponding to $B_{\mathrm{tr}}|_{\mathfrak{t} \times \mathfrak{t}}$.
	
	When $i \neq j$, we have $\mathfrak{g}(W_i, h_i) \mathfrak{g}(W_j, h_j) = \{0\}$ inside the ring $\End(W)$, because $\mathfrak{g}(W_i, h_i)$ is zero on $W_j$. Hence the subspaces $\mathfrak{g}(W_i, h_i)$ and $\mathfrak{g}(W_j, h_j)$ of $\mathfrak{g}$ are orthogonal relative to $B_{\mathrm{tr}}$. It follows that $Y_i$ and $Y_j$ are orthogonal relative to $Q'$, as desired.
\end{proof}

\begin{theorem}\label{prop:main}
	Let $G = G(W, h)$ be a classical group over $F$ in one of the following cases:
	\begin{itemize}
		\item $(W, h)$ is symplectic,
		\item $(W, h)$ is quadratic and $\dim W$ is odd,
		\item $(W, h)$ is quadratic and $\dim W \geq 6$ is even.
	\end{itemize}
	Let $\tilde{\mathcal{G}}$ be an object of $\cate{CExt}(G, \mathbf{K}_2)$. Then for every Levi subgroup $M$ of $G$, the pull-back $\tilde{\mathcal{M}}$ of $\tilde{\mathcal{G}}$ to $M$ is the external Baer sum of its pull-backs to each factor in \eqref{eqn:Levi}.
\end{theorem}
\begin{proof}
	First off, we may exclude the quadratic $F$-vector spaces with $\dim \leq 2$, in which case $G$ is a torus.
	
	Fix a maximal $F$-torus $T$ of $M$. Let $Y = \bigoplus_{i=1}^r Y_i \oplus Y^\flat$ be the decomposition of $Y := \mathrm{X}_*(T_{\overline{F}})$ relative to \eqref{eqn:Levi}.
	
	It suffices to decompose the object in $\cate{BD}(M, T)$ attached to $\tilde{\mathcal{M}}$, by Corollary \ref{prop:BD-external}. Note that by the second part of Theorem \ref{prop:BD}, that object is pulled back from the $(Q, \mathcal{D}, \varphi) \in \cate{BD}(G, T)$ attached to $\tilde{\mathcal{G}}$.

	Proposition \ref{prop:Baer-orthogonal} reduces us to showing that $Y_1, \ldots, Y_r, Y^\flat$ are mutually orthogonal relative to $B_Q$. Denote by $\mathbb{H}$ the hyperbolic plane (resp.\ any two-dimensional symplectic vector space) in the quadratic (resp.\ symplectic) case. Then the desired orthogonality follows from Lemma \ref{prop:main-aux}, namely by taking $s = r+1$, $W_s = W^\flat$ and $W_i = \mathbb{H}^{\oplus d_i}$ for each $1 \leq i \leq r$, and then base-change to $\overline{F}$; note that $M \subset \prod_{i=1}^s G(W_i, h_i)$.
\end{proof}

\begin{remark}
	The arguments above yield the following variant: let $W = \bigoplus_{i=1}^s W_i$ and $G = G(W, h)$ be as in Lemma \ref{prop:main-aux}, but without assuming $F = \overline{F}$, then the pull-back of every $\tilde{\mathcal{G}} \in \cate{CExt}(G, \mathbf{K}_2)$ to $\prod_{i=1}^s G(W_i, h_i) \subset G$ breaks into an external Baer sum. 
\end{remark}

We now assume $F$ is local, take $m \in \Z_{\geq 1}$ and consider a classical group $G = G(W, h)$ over $F$ in general. Specialize the formalism of \S\ref{sec:cext} to central extensions
\[ 1 \to \bmu_m \to \tilde{G} \to G(F) \to 1 \]
in $\cate{CExt}(G(F), \bmu_m)$. Objects of $\cate{CExt}(G(F), \bmu_m)$ are called $m$-fold covers of $G(F)$.

Given a Levi subgroup $M \subset G$ as in \eqref{eqn:Levi}, we have
\begin{equation}\label{eqn:Levi-F}
	M(F) = \prod_{i=1}^r \GL_{d_i}(E) \times G^\flat(F);
\end{equation}
we denote the preimages of $M(F)$, $G^\flat(F)$ and $\GL_{d_i}(E)$ in $\tilde{G}$ by $\tilde{M}$, $\tilde{G}^\flat$ and $\widetilde{\GL_{d_i}(E)}$, respectively. We also say that $\tilde{M}$ is a Levi subgroup of $\tilde{G}$. Below is a paraphrase of the definition in \cite[\S 4]{HW25}.

\begin{definition}\label{def:decomposed-Levi}
	Let $\tilde{M}$ be a Levi subgroup of $\tilde{G}$, described as above. We say $\tilde{M}$ is decomposed if
	\[ \tilde{M} \simeq \widetilde{\GL_{d_1}(E)} \dot{\boxplus} \cdots \dot{\boxplus} \widetilde{\GL_{d_r}(E)} \dot{\boxplus} \tilde{G}^\flat \]
	in $\cate{CExt}(M(F), \bmu_m)$.
\end{definition}

\begin{proposition}\label{prop:decompose-comm}
	The Levi subgroup $\tilde{M}$ of $\tilde{G}$ is decomposed if and only if the commutator map $\mathrm{comm}$ is trivial on each pair of factors of \eqref{eqn:Levi-F}.
\end{proposition}
\begin{proof}
	Apply Proposition \ref{prop:recognition-Baer}.
\end{proof}

\begin{corollary}\label{prop:main-topological}
	Let $F$ be a local field, and let $m \mid m_F$ be as in \S\ref{sec:BD}. Suppose that $1 \to \bmu_m \to \tilde{G} \to G(F) \to 1$ is an $m$-fold BD-cover, where $G = G(W, h)$ is subject to the same constraints in Theorem \ref{prop:main}. If $M$ is any Levi subgroup of $G$, then $\tilde{M}$ is decomposed.
\end{corollary}
\begin{proof}
	Let $\tilde{\mathcal{G}}$ be the object of $\cate{CExt}(G, \mathbf{K}_2)$ giving rise to $\tilde{G}$. Let $\tilde{\mathcal{M}}$ be the pull-back of $\tilde{\mathcal{G}}$ to $M$, which gives rise to the BD-cover $\tilde{M}$ (Proposition \ref{prop:F-pts-Baer}).
	
	To decompose $\tilde{M}$ into the desired external Baer sum, it suffices to decompose $\tilde{\mathcal{M}}$ by the second part of Proposition \ref{prop:F-pts-Baer}. Now apply Theorem \ref{prop:main}.
\end{proof}

We remark that when $F$ is non-Archimedean and $G$ is symplectic, the BD-covers exhaust all finite topological central extensions of $G(F)$; see the discussions in \cite[\S 3]{GGW}.

\section{Counterexamples for unitary groups}\label{sec:counterexample-unitary}
We will give examples of non-decomposed Levi subgroups in BD-covers of unitary groups. Retain the notations from \S\ref{sec:classical-group}.

Consider even unitary groups first. Fix $N \in \Z_{\geq 2}$ and a quadratic extension $E|F$ of fields with $\mathrm{char}(F) \neq 2$. Denote by $x \mapsto \overline{x}$ the non-identity element of $\Gal(E|F)$ (where $x \in E$).

Given $c \in \Z$, take
\[ 1 \to \mathbf{K}_2 \to \widetilde{\mathcal{GL}}_N^{(c)} \to \GL_N \to 1 \]
to be the central $\mathbf{K}_2$-extension that gives rise to Kazhdan--Patterson covers \cite{KP84} with twisting parameter $c$; see \cite[\S 2.7.7]{Weis18}. Let $T_N$ be the standard maximal torus of $\GL_N$ with cocharacter lattice $Y_N = \Z^N$. Then the corresponding datum $\left( Q_N^{(c)}, \mathcal{D}_N^{(c)}, \varphi_N^{(c)} \right) \in \cate{BD}(\GL_N, T_N)$ satisfies
\begin{equation}\label{eqn:QN}
	Q_N^{(c)} = (1+c) \left( \sum_{i=1}^N \mathsf{X}_i \right)^2 - \sum_{1 \leq i < j \leq N} \mathsf{X}_i \mathsf{X}_j;
\end{equation}
see \textit{loc.\ cit.}

We view the above as a central $\mathbf{K}_2$-extension over $E$, then apply the Weil restriction along $E|F$ for central $\mathbf{K}_2$-extensions explained in \cite[\S 2.3]{Li20}. This yields a central $\mathbf{K}_2$-extension over $F$, namely
\begin{equation}\label{eqn:KP-K2}
	0 \to \mathbf{K}_2 \to \Res_{E|F} \widetilde{\mathcal{GL}}_N^{(c)} \to \Res_{E|F} \GL_N \to 1.
\end{equation}

Now specialize to the case
\[ N := 2n, \quad n \in \Z_{\geq 2}. \]
Take $(W, h)$ to be the sum of $n$ copies of the hyperbolic Hermitian form, thus an Hermitian form of dimension $2n$ for $E|F$. Denote by $G := G(W, h)$ the corresponding unitary group, embedded into $\Res_{E|F} \GL_{2n}$ via the standard (hyperbolic) basis. Therefore $T := (\Res_{E|F} T_{2n}) \cap G$ is a maximal $F$-torus of $G$, consisting of matrices
\begin{equation}\label{eqn:T-unitary}
	\begin{pmatrix}
		x_1 & & & & & \\
		& \ddots & & & & \\
		& & x_n & & & \\
		& & & \overline{x_n}^{-1} & & \\
		& & & & \ddots & \\
		& & & & & \overline{x_1}^{-1}
	\end{pmatrix}, \quad
	x_1, \ldots, x_n \in E^{\times}.
\end{equation}

Pulling \eqref{eqn:KP-K2} back to $G := G(W, h)$, we obtain
\begin{equation}\label{eqn:unitary-K2}
	0 \to \mathbf{K}_2 \to \tilde{\mathcal{G}} \to G \to 1,
\end{equation}
an object of $\cate{CExt}(G, \mathbf{K}_2)$. Note that $T \simeq (\Res_{E|F} \Gm)^n$ is also a minimal Levi subgroup of $G$.

The result below shows that Theorem \ref{prop:main} does not hold for unitary groups and their Levi subgroups in general.

\begin{proposition}\label{prop:unitary-example}
	Let $\tilde{\mathcal{T}}$ be the pull-back of \eqref{eqn:unitary-K2} to $T$. Then $\tilde{\mathcal{T}}$ is not the external Baer sum of its pull-backs to the $n$ factors $\Res_{E|F} \Gm$.
\end{proposition}
\begin{proof}
	Take the maximal $F$-torus $\Res_{E|F}(T_{2n})$ of $\Res_{E|F} \GL_{2n}$, whose cocharacter lattice is two copies of $Y_{2n}$ exchanged by $\Gal(E|F)$. The datum $(Q', \mathcal{D}', \varphi')$ attached to $\Res_{E|F} \widetilde{\mathcal{GL}}_{2n}^{(c)}$ is determined in \cite[Theorem 2.3.10]{Li20}. What matters here is the part $Q'$: by \textit{loc.\ cit.}, it is the sum of \eqref{eqn:QN} and its avatar in the ``conjugate variables'' $\mathsf{Y}_1, \ldots, \mathsf{Y}_{2n}$, namely
	\[ Q' = (1 + c) \left( \left( \sum_{i=1}^{2n} \mathsf{X}_i \right)^2 + \left( \sum_{i=1}^{2n} \mathsf{Y}_i \right)^2 \right) - \sum_{1 \leq i < j \leq 2n} (\mathsf{X}_i \mathsf{X}_j + \mathsf{Y}_i \mathsf{Y}_j). \]
	
	We pull $\Res_{E|F} \widetilde{\mathcal{GL}}_{2n}^{(c)}$ back to $G$ to get $(Q, \mathcal{D}, \varphi) \in \cate{BD}(G, T)$, using the second part of Theorem \ref{prop:BD}. In view of \eqref{eqn:T-unitary}, this amounts to substituting $\mathsf{X}_{n+1}, \ldots, \mathsf{X}_{2n}$ by $-\mathsf{Y}_n, \ldots, - \mathsf{Y}_1$, and $\mathsf{Y}_{n+1}, \ldots, \mathsf{Y}_{2n}$ by $-\mathsf{X}_n, \ldots, -\mathsf{X}_1$ accordingly in $Q'$. Therefore
	\begin{multline*}
		Q = 2(1 + c) \left( \mathsf{X}_1 + \cdots + \mathsf{X}_n - \mathsf{Y}_1 - \cdots - \mathsf{Y}_n \right)^2 \\
		- 2\sum_{1 \leq i < j \leq n} (\mathsf{X}_i \mathsf{X}_j + \mathsf{Y}_i \mathsf{Y}_j) + 2\sum_{1 \leq i, j \leq n} \mathsf{X}_i \mathsf{Y}_j.
	\end{multline*}
	
	Observe that $\Gal(E|F)$ acts on $Y := \mathrm{X}_*(T_E)$ by swapping $\mathsf{X}_i$ and $\mathsf{Y}_i$ for all $i$. In fact $Y$ is the induced Galois module $\Ind_{E|F} (\Z^n) := \Ind_{\Gal_E}^{\Gal_F} (\Z^n)$.
	
	For every $i$, the pair of variables $\mathsf{X}_i, \mathsf{Y}_i$ lives over the $i$-th factor $\Res_{E|F} \Gm$ of $T$. Hence $Q$ is the sum of ``diagonal'' terms plus the following ``crossed'' terms:
	\begin{equation}\label{eqn:KP-Q-crossed}
		\begin{gathered}
			(4(1+c) - 2) (\mathsf{X}_i \mathsf{X}_j + \mathsf{Y}_i \mathsf{Y}_j), \quad
			-(4(1+c) - 2) (\mathsf{X}_i \mathsf{Y}_j  + \mathsf{X}_j \mathsf{Y}_i) \\
			\text{for all}\; 1 \leq i < j \leq n.
		\end{gathered}
	\end{equation}
	
	Since $4(1+c) - 2 \neq 0$, it follows that $\tilde{\mathcal{T}}$ cannot be the external Baer sum of pull-backs to these $n$ copies of $\Res_{E|F} \Gm$.
\end{proof}

Now assume $F$ is local. For every $m \in \Z_{\geq 1}$ with $m \mid m_F$ and an embedding $\epsilon: \mu_m(F) \hookrightarrow \CC^{\times}$, from \eqref{eqn:unitary-K2} one obtains an $m$-fold BD-cover $\tilde{G}$ of the unitary group $G(F)$, as well as its pull-back $\tilde{T}$ to $T(F)$ which arises from $\tilde{\mathcal{T}}$. The following result shows that Corollary \ref{prop:main-topological} does not hold for BD-covers of unitary groups in general. 

\begin{proposition}\label{prop:unitary-example-F}
	In the setting above, assume furthermore $m = m_F$ and that
	\begin{itemize}
		\item $F$ is a finite extension of $\Q_p$ for some odd prime $p$;
		\item $|4c + 2| = dd'$ where $d, d' \in \Z_{\geq 1}$ and $d$ (resp.\ $d'$) divides $m$ (resp.\ is coprime to $m$);
		\item $p$ divides $m/d$.
	\end{itemize}
	
	Then the Levi subgroup $\tilde{T}$ of $\tilde{G}$ is not decomposed.
\end{proposition}
\begin{proof}
	Denote by $X$ the dual of $Y$, viewed as a Galois module. There is a Galois-equivariant homomorphism $\alpha: Y \to X$ induced by $B_Q$; we get
	\[ \alpha \otimes \mu_m: Y \otimes \mu_m \to X \otimes \mu_m. \]
	Note that $X \otimes \mu_m$ is canonically isomorphic to $\Ind_{E|F}\left( \Z^n \otimes \mu_m \right)$; see e.g.\ \cite[(4.6)]{MS24}. The same also holds for $Y \otimes \mu_m$. We can view $\alpha \otimes \mu_m$ as a symmetric $n \times n$ matrix $(f_{ij})_{i, j}$ over $\End_{\Z}(\Ind_{E|F}(\mu_m))$. Each $f_{ij}$ is a $\Gal_F$-equivariant endomorphism since $\alpha \otimes \mu_m$ is.
	
	Now apply \cite[\S 1.3]{Weis16b} to $T$: Tate's local duality gives a perfect pairing
	\begin{equation*}
		\mathrm{Tate}: \Hm^1(F, Y \otimes \mu_m) \times \Hm^1(F, X \otimes \mu_m) \to \mu_m(F) \rightiso \bmu_m.
	\end{equation*}
	
	The Kummer sequence gives $T(F)/T(F)^m \rightiso \Hm^1(F, Y \otimes \mu_m)$ since $T$ is induced, and by \cite[(1.8)]{Weis16b}, $\mathrm{comm}: T(F) \times T(F) \to \bmu_m$ descends to the anti-symmetric form on $\Hm^1(F, Y \otimes \mu_m)$ obtained from
	\begin{equation*}
		\Hm^1(F, Y \otimes \mu_m) \times \Hm^1(F, Y \otimes \mu_m) \xrightarrow{(\identity, \alpha \otimes \mu_m)}
		\Hm^1(F, Y \otimes \mu_m) \times \Hm^1(F, X \otimes \mu_m) \xrightarrow{\text{Tate}} \bmu_m.
	\end{equation*}
	
	For all $1 \leq i \leq n$, denote by $T_i$ the $i$-th factor $\Res_{E|F} \Gm$ in $T$. Fix $i < j$. Identify both the $i$-th and $j$-th piece of $Y$ and $X$ with $\Ind_{E|F} (\mu_m)$, then by \textit{loc.\ cit.}, $\mathrm{comm}|_{T_i(F) \times T_j(F)}$ comes from the composite
	\begin{multline*}
		\Hm^1(F, \Ind_{E|F}(\mu_m)) \times \Hm^1(F, \Ind_{E|F}(\mu_m)) \xrightarrow{(\identity, f_{ij})} \\
		\Hm^1(F, \Ind_{E|F}(\mu_m)) \times \Hm^1(F, \Ind_{E|F}(\mu_m)) \xrightarrow{\text{Tate}} \bmu_m.
	\end{multline*}
	
	Recall that $f_{ij}$ is induced by the terms in $Q$ (or: in $B_Q$) with indices $i < j$. Some observations are in order.
	\begin{itemize}
		\item Shapiro's lemma implies $\Hm^1(F, \Ind_{E|F}(\mu_m)) \simeq \Hm^1(E, \mu_m) \simeq E^{\times}/E^{\times m}$.
		\item For $\Hm^1(F, \Ind_{E|F}(\mu_m))$, the pairing $\mathrm{Tate}$ is then identified with $(\cdot, \cdot)_{E, m}$; this is \cite[(4.15)]{MS24} in the global case, and the local case is similar and easier.
		\item In view of \eqref{eqn:KP-Q-crossed}, $f_{ij}$ is identified with the endomorphism of $E^{\times}/E^{\times m}$ induced by
		\[ w \mapsto (w\overline{w}^{-1})^{4(1+c) - 2}. \]
		Indeed, $w$ (resp.\ $\overline{w}$) on the right comes from $\mathsf{X}_i \mathsf{X}_j + \mathsf{Y}_i \mathsf{Y}_j$ (resp.\ $\mathsf{X}_i \mathsf{Y}_j + \mathsf{X}_j \mathsf{Y}_i$).
	\end{itemize}
	
	Therefore $\mathrm{comm}|_{T_i(F) \times T_j(F)}$ becomes
	\[ (z, w) \mapsto (z, w\overline{w}^{-1})^{4c + 2}_{E, m}, \quad z, w \in E^{\times}. \]
	We will discuss the constancy of the map above in $(z, w)$. When $4c + 2 < 0$, replacing $w$ by $\overline{w}$ changes $(z, w\overline{w}^{-1})_{E,m}^{4c+2}$ into $(z, w\overline{w}^{-1})_{E,m}^{dd'}$; constancy is unaffected. Thus we may assume $4c + 2 = dd' \geq 1$ in what follows.
	
	The following statements are equivalent:
	\begin{enumerate}[(i)]
		\item $\mathrm{comm}|_{T_i(F) \times T_j(F)}$ is constant;
		\item $(z, w) \mapsto (z^d, w\overline{w}^{-1})_{E, m}$ is constant;
		\item $(z, w) \mapsto (z, w\overline{w}^{-1})_{E, m/d}$ is constant, see e.g.\ \cite[Remark 2.4.1]{Li20} for the equivalence with (ii);
		\item $u \in E^{\times m/d}$ for all norm-one elements $u \in E^{\times}$, by the non-degeneracy of Hilbert symbols.
	\end{enumerate}
	
	To show $\tilde{T}$ is not decomposed, by Proposition \ref{prop:recognition-Baer} it suffices to disprove (i). Since $p \mid \frac{m}{d}$, we will disprove (iv) by exhibiting a norm-one element $u \in U^{(1)} \smallsetminus E^{\times p}$, where $U^{(1)} := 1 + \mathfrak{p}_E$. Observe that if $v \in E^{\times}$ satisfies $v^p \in U^{(1)}$, then $v \in U^{(1)}$.

	Recall (see \cite[Chapter 2, (5.7)]{Ne99}) that: (a) $U^{(1)}$ is actually a finitely generated $\Z_p$-module; (b) its torsion submodule is $\mu_{p^a}(E)$ for some $a \geq 0$; (c) its torsion-free quotient module $\mathcal{F}$ satisfies $\mathcal{F} \simeq \Z_p^{[E : \Q_p]}$ as topological modules. The $\Gal(E|F)$-action on $U^{(1)}$ passes to $\mathcal{F}$, and decomposes it into eigen-submodules $\mathcal{F}^+ \oplus \mathcal{F}^-$ since $2 \in \Z_p^{\times}$.
	
	We must have $\mathcal{F}^- \neq 0$, otherwise $E^{\times}$ would have only finitely many norm-one elements, which is absurd. Since $\mathcal{F}^- \subset \mathcal{F}$, it is free of finite rank over $\Z_p$. Take $[u] \in \mathcal{F}^- \smallsetminus p \mathcal{F}^-$ (written additively). Hence $[u] \notin p \mathcal{F}$ by $\mathcal{F} = \mathcal{F}^+ \oplus \mathcal{F}^-$.
	
	For every preimage $u \in U^{(1)}$ of $[u]$, we have $\overline{u} = \zeta u^{-1}$ where $\zeta \in \mu_{p^a}(E)$; it follows that $\overline{\zeta} = \zeta$. Since $\mu_{p^a}(E)$ is a $\Gal(E|F)$-module of odd order, we have
	\[ \forall i \geq 1, \; \Hm^i(\Gal(E|F), \mu_{p^a}(E)) = 0. \]
	Hence $\zeta = \xi\overline{\xi}$ for some $\xi \in \mu_{p^a}(E)$. Replacing $u$ by $u\xi^{-1}$, we may assume $\overline{u} = u^{-1}$. If $u = v^p$ for some $v \in U^{(1)}$, then $[u] = p[v] \in p\mathcal{F}$. Contradiction.
\end{proof}

Note that given an odd prime $p$, there is always some finite extension $F \supset \Q_p$ such that $p \mid m_F$. Since $2 \mid m_F$, one can always find $c$ that meets the conditions in Proposition \ref{prop:unitary-example-F}, say with $d = 2$.

Consider quasisplit odd unitary groups next. Restore $E|F$ to a quadratic extension of fields with $\mathrm{char}(F) \neq 2$. Specialize the constructions in \eqref{eqn:QN}--\eqref{eqn:KP-K2} to
\[ N := 2n-1, \quad n \in \Z_{\geq 2}. \]

Let $L$ be a one-dimensional non-degenerate Hermitian space for $E|F$, and take the orthogonal sum
\[ W := \mathbb{H}^{\oplus(n-1)} \oplus L, \]
where $\mathbb{H}$ is the hyperbolic Hermitian plane. Embed $G := G(W, h)$ into $\Res_{E|F}\GL_N$ via the standard basis. This gives all quasisplit odd unitary groups, and $G$ contains the maximal $F$-torus
\[ T \simeq (\Res_{E|F}\Gm)^{n-1} \times G(L) \]
consisting of matrices
\[ \operatorname{diag}(x_1, \ldots, x_{n-1}, u, \overline{x}_{n-1}^{-1}, \ldots, \overline{x}_1^{-1}), \]
where $x_1, \ldots, x_{n-1} \in E^{\times}$ and $u \in E^1 := \ker(\mathrm{N}_{E|F})$; it is also a minimal Levi subgroup of $G$.

Pulling back the Weil restriction $\Res_{E|F} \widetilde{\mathcal{GL}}_{N}^{(c)}$ to $G$ yields a central $\mathbf{K}_2$-extension $\tilde{\mathcal{G}}$, as in \eqref{eqn:unitary-K2}.

\begin{proposition}\label{prop:odd-unitary-example}
	Let $\tilde{\mathcal{T}}$ be the pull-back of $\tilde{\mathcal{G}}$ to $T$. Then $\tilde{\mathcal{T}}$ is not the external Baer sum of its pull-backs to the $n-1$ copies of $\Res_{E|F} \Gm$ and to $G(L)$.
\end{proposition}
\begin{proof}
	Write
	\[ Y = \bigoplus_{i=1}^{n-1}(\Z\mathsf{X}_i \oplus \Z\mathsf{Y}_i) \oplus \Z\mathsf{Z}, \]
	where $\Gal(E|F)$ swaps $\mathsf{X}_i$ and $\mathsf{Y}_i$ and sends $\mathsf{Z}$ to $-\mathsf{Z}$. The substitutions from the proof of Proposition \ref{prop:unitary-example} now read
	\begin{gather*}
		\mathsf{X}_{N+1-i} = -\mathsf{Y}_i, \quad \mathsf{Y}_{N+1-i} = -\mathsf{X}_i \quad (1 \leq i \leq n-1), \\
		\mathsf{X}_n = \mathsf{Z}, \quad \mathsf{Y}_n = -\mathsf{Z}.
	\end{gather*}
	By the symmetry of $Q_N^{(c)}$, the resulting quadratic form is seen to be
	\begin{equation*}
		Q = 2Q_N^{(c)}(\mathsf{X}_1, \ldots, \mathsf{X}_{n-1}, \mathsf{Z}, -\mathsf{Y}_{n-1}, \ldots, -\mathsf{Y}_1).
	\end{equation*}
	It follows from \eqref{eqn:QN} that the crossed terms between $\mathsf{Z}$ and the $i$-th hyperbolic factor are
	\begin{equation}\label{eqn:odd-unitary-crossed}
		(4c+2) \mathsf{Z}(\mathsf{X}_i-\mathsf{Y}_i), \qquad 1 \leq i \leq n-1.
	\end{equation}
	In particular, $B_Q(\mathsf{X}_i, \mathsf{Z}) = 4c+2 \neq 0$, and Proposition \ref{prop:Baer-orthogonal} yields the assertion.
\end{proof}

Now assume $F$ is local. For $m$ and $\epsilon$ as in the even case, $\tilde{\mathcal{G}}$ gives rise to an $m$-fold BD-cover $\tilde{G}$ and its pull-back $\tilde{T}$ to $T(F)$.

\begin{proposition}\label{prop:odd-unitary-example-F}
	Assume that the hypotheses of Proposition \ref{prop:unitary-example-F} hold in the setting above. Then $\tilde{T} \subset \tilde{G}$ is not decomposed.
\end{proposition}
\begin{proof}
	Put $a = 4c+2$. For all $1 \leq i \leq n-1$, let $T_i$ be the $i$-th factor in $T$; identify $T_i(F)$ (resp.\ $G(L)(F)$) with $E^{\times}$ (resp.\ $E^1$). Realize $G(L)$ as the quotient of the induced torus $\Res_{E|F} \Gm$ by Hilbert 90, namely via $w \mapsto w\overline{w}^{-1}$. Accordingly, replace the $\mathsf{Z}$ in the proof of Proposition \ref{prop:odd-unitary-example} by $\mathsf{X} - \mathsf{Y}$, where $\mathsf{Y}$ is the ``conjugate variable'' of $\mathsf{X}$ as usual. Substituting into \eqref{eqn:odd-unitary-crossed}, we obtain
	\[ (4c+2) \left( (\mathsf{X}\mathsf{X}_i + \mathsf{Y}\mathsf{Y}_i) - (\mathsf{X}\mathsf{Y}_i + \mathsf{X}_i \mathsf{Y}) \right), \]
	which has the same form as the crossed terms associated with a pair of indices in \eqref{eqn:KP-Q-crossed}. Therefore, by repeating the Tate-duality argument in the proof of Proposition \ref{prop:unitary-example-F}, we obtain
	\begin{equation}\label{eqn:odd-unitary-comm}
		\mathrm{comm}(z, u) = (z, u)_{E, m}^a, \quad z \in E^{\times} = T_i(F), \quad u \in E^1 = G(L)(F).
	\end{equation}
	
	The remainder of the proof of Proposition \ref{prop:unitary-example-F} now applies with $w\overline{w}^{-1}$ replaced by $u$, and then Proposition \ref{prop:decompose-comm} shows that $\tilde{T}$ is not decomposed.
\end{proof}

The choices of $F$, $m$, $c$, $d$ and $d'$ discussed after Proposition \ref{prop:unitary-example-F} also furnish examples in the odd unitary setting.

\begin{remark}\label{rem:Savin-cover}
	Let $G = G(W, h)$ be a unitary group with respect to $E|F$, where $N := \dim_E W \in \Z_{\geq 1}$. In contrast with the previous constructions based on Kazhdan--Patterson covers, one can also obtain BD-covers of $G$ by pulling back the \emph{Savin covers} of $\GL_N(E)$. Recall that the Savin covers are defined by pulling the standard $m$-fold cover of $\Sp_{2N}(E)$ (see \cite[\S 2.7.4]{Weis18}) back to $\GL_N(E)$. Let us show that Levi subgroups always decompose in such BD-covers of $G$.
	
	Indeed, Levi subgroups of $G$ are intersections of Levi subgroups of $\Res_{E|F} \GL_N$ with $G$, and it suffices to argue that Levi subgroups decompose in Savin covers. On the other hand, Levi subgroups in the Siegel Levi subgroup $\GL_N$ of $\Sp_{2N}$ always arise from an orthogonal decomposition of the $2N$-dimensional symplectic vector space, hence arise by intersecting with Levi subgroups of $\Sp_{2N}$. We infer the relevant commutation at the $\mathbf{K}_2$-level by applying Theorem \ref{prop:main} to $\Sp_{2N}$ over $E$; pullbacks and Weil restriction preserve commutation. Remark \ref{rem:recognition-Baer} and Proposition \ref{prop:F-pts-Baer} then yield the desired decomposition after passage to BD-covers.
\end{remark}

\section{The case of \texorpdfstring{$\SO(4)$}{SO(4)}}\label{sec:counterexample-SO4}
Examples of non-decomposed Levi subgroups in BD-covers of $\SO(4)$ will be given in this section.

Assume that $F$ is any field with $\mathrm{char}(F) \neq 2$. Let $(V, q)$ be a quadratic $F$-vector space of dimension $4$, and take the special orthogonal group $G := G(V, q)$.

The anisotropic kernel of $(V, q)$ has dimension $0$, $2$ or $4$. Only the first two cases need consideration.

First, suppose that $V$ is the sum of two copies of the hyperbolic plane $\mathbb{H}$. Then $G$ has a split maximal $F$-torus $T \simeq \Gm^2$: each $\Gm$ is isomorphic to $\SO(\mathbb{H})$. The Weyl group $\mathrm{W}(G, T)$ acting on $Y := \mathrm{X}_*(T) \simeq \Z^2$ is generated by the automorphisms
\[ (x, y) \mapsto (y, x) \;\text{and}\; (x, y) \mapsto (-x, -y), \quad (x, y) \in \Z^2. \]

The $\mathrm{W}(G, T)$-invariant quadratic forms $Y \to \Z$ are exactly those of the form
\begin{equation}\label{eqn:Q-SO4}
	Q = a(\mathsf{X}_1^2 + \mathsf{X}_2^2) + b \mathsf{X}_1 \mathsf{X}_2, \quad a, b \in \Z.
\end{equation}

One can always extend the $Q$ in \eqref{eqn:Q-SO4} to some $(Q, \mathcal{D}, \varphi) \in \cate{BD}(G, T)$. Indeed, by \cite[\S 2.6]{GG}, there exists a bilinear map $D: Y \times Y \to \Z$ (called a bisector) such that $D(y, y) = Q(y)$, and by choosing a group homomorphism $\eta: Y_{\mathrm{sc}} \to F^{\times}$ one obtains $(D, \eta) \in \cate{Bis}_{G, Q}$ (see \textit{loc.\ cit.}) that ``incarnates'' an object $(Q, \mathcal{D}, \varphi) \in \cate{BD}(G, T)$. We obtain a corresponding central extension
\begin{equation}\label{eqn:SO4-K2}
	0 \to \mathbf{K}_2 \to \tilde{\mathcal{G}} \to G \to 1.
\end{equation}

\begin{proposition}\label{prop:SO4-example}
	For $G = G(V, q)$ and $T$ as above, let $\tilde{\mathcal{T}}$ be the pull-back of \eqref{eqn:SO4-K2} to $T$. If $b \neq 0$, then $\tilde{\mathcal{T}}$ is not the external Baer sum of its pull-backs to the two factors $\Gm$.
\end{proposition}
\begin{proof}
	Indeed, $Q$ contains the ``crossed'' term $\mathsf{X}_1 \mathsf{X}_2$ with nonzero coefficient.
\end{proof}

\begin{remark}\label{rem:SO4-example-F}
	Suppose $F$ is local. Choose $m \mid m_F$ and take the $m$-fold BD-cover $\tilde{G}$ of $G(F)$ arising from \eqref{eqn:SO4-K2}. As in Proposition \ref{prop:unitary-example-F}, but with much simpler arguments (using the commutator formula \cite[Corollary 3.14]{BD01} instead of Tate duality), one can show that for the preimage $\tilde{T}$ of $T(F)$, the commutator pairing between two copies of $F^{\times}$ is $(z, w) \mapsto (z, w)_{F, m}^b$. When $m \nmid b$, the commutator is non-constant and $\tilde{T}$ is not decomposed.
\end{remark}

Therefore, both Theorem \ref{prop:main} and Corollary \ref{prop:main-topological} fail for split $\SO(4)$.

Secondly, suppose that $V = \mathbb{H} \oplus V_{\mathrm{ani}}$, where $V_{\mathrm{ani}}$ is an anisotropic subspace of dimension $2$. In this case $G$ has a maximal $F$-torus
\[ T \simeq \Gm \times G(V_{\mathrm{ani}}), \]
which is also the only proper Levi subgroup of $G$ up to conjugacy. We then have the following counterpart of Theorem \ref{prop:main}.

\begin{proposition}\label{prop:SO4-example-nonsplit}
	For $G = G(V, q)$ and $T$ as above, let $\tilde{\mathcal{G}} \in \cate{CExt}(G, \mathbf{K}_2)$ and let $\tilde{\mathcal{T}}$ be its pull-back to $T$. Then $\tilde{\mathcal{T}}$ is the external Baer sum of its pull-backs to $\Gm$ and $G(V_{\mathrm{ani}})$.
	
	Consequently, when $F$ is local, Levi subgroups are decomposed in all BD-covers of $G(F)$.
\end{proposition}
\begin{proof}
	Write $Y := \mathrm{X}_*(T)$. As $\Z$-modules, $Y \simeq \Z^2$ where the two coordinates correspond to the factors $\Gm$ and $G(V_{\mathrm{ani}})$ of $T$. Take a quadratic extension $F'|F$ that splits $G(V_{\mathrm{ani}})$.
	
	Let $(Q, \mathcal{D}, \varphi) \in \cate{BD}(G, T)$ be attached to $\tilde{\mathcal{G}}$. It remains to show that compared with \eqref{eqn:Q-SO4}, the only choice of $Q$ here is $a(\mathsf{X}_1^2 + \mathsf{X}_2^2)$, i.e.\ there are no crossed terms. Indeed, $\Gal(F'|F)$ fixes $\mathsf{X}_1$ and negates $\mathsf{X}_2$, whilst $Q$ must be $\Gal(F'|F)$-invariant. This completes the proof.
\end{proof}

\section{On the full orthogonal groups}\label{sec:full-orthogonal}
The case of full orthogonal groups is left unaddressed in \S\ref{sec:classical-group} and \S\S\ref{sec:counterexample-unitary}--\ref{sec:counterexample-SO4}. This section will offer a partial remedy.

Brylinski--Deligne theory applies only to connected reductive groups, so we proceed as follows. Consider a nonzero quadratic $F$-vector space $(V, q)$ and let $G^+ := \Or(V)$ be the full orthogonal group. Levi subgroups $M^+$ of $G^+$ have a form similar to \eqref{eqn:Levi}, namely
\[ M^+ = \prod_{i=1}^r \GL_{d_i} \times \Or(V^\flat), \quad r \in \Z_{\geq 0}, \quad d_i \in \Z_{\geq 1}. \]
Here we demand that $V^\flat$ is nonzero. Let $G$ (resp.\ $M$) be the identity connected component of $G^+$ (resp.\ $M^+$).

Let $m \in \Z_{\geq 1}$. Consider a topological central extension
\[ 1 \to \bmu_m \to \widetilde{G^+} \xrightarrow{\rev} G^+(F) \to 1. \]
Set $\widetilde{M^+} := \rev^{-1}(M^+(F))$. Denote by $\tilde{G}$ (resp.\ $\tilde{M}$) the preimages of $G(F)$ (resp.\ $M(F)$) in $\widetilde{G^+}$. Pick any $F$-point $\tau \in \Or(V^\flat) \smallsetminus \SO(V^\flat)$, viewed as an element of $M^+(F)$ that commutes with elements from each $\GL_{d_i}(F)$. Then $M^+ = M \sqcup M\tau$.

Definition \ref{def:decomposed-Levi} naturally extends to the setting of $\widetilde{M^+} \subset \widetilde{G^+}$.

\begin{proposition}\label{prop:O}
	In the setting above, $\widetilde{M^+}$ is decomposed if and only if
	\begin{enumerate}[(i)]
		\item $\tilde{M}$ is decomposed, and
		\item $\mathrm{comm}(t_i, \tau) = 1$ for all $1 \leq i \leq r$ and all $t_i$ in the standard maximal torus of $\GL_{d_i}(F)$.
	\end{enumerate}
\end{proposition}
\begin{proof}
	In view of Proposition \ref{prop:recognition-Baer}, this will be immediately clear if (ii) is replaced by (ii)': $\mathrm{comm}(g_i, \tau) = 1$ for all $1 \leq i \leq r$ and $g_i \in \GL_{d_i}(F)$. Nonetheless, by the canonical splitting of topological central extensions over unipotent radicals of parabolics, we have $\mathrm{comm}(u_i, \tau) = 1$ for all $u_i$ in the standard maximal unipotent subgroup of $\GL_{d_i}(F)$ or in its opposite. Hence (ii) implies $\mathrm{comm}(g_i, \tau) = 1$ for all $g_i \in \GL_{d_i}(F)$ in the open Bruhat cell, so (ii) implies (ii)'.
\end{proof}

The condition (i) can be studied using the results for special orthogonal groups in \S\ref{sec:classical-group} and \S\ref{sec:counterexample-SO4}. On the other hand, when the conjugation action of $\tau$ lifts to the $\mathbf{K}_2$-level and $T \subset M$ is a $\tau$-stable maximal $F$-torus, condition (ii) can be checked through the induced action of $\tau$ on $(Q, \mathcal{D}, \varphi) \in \cate{BD}(G, T)$.

\section{On odd \texorpdfstring{$\GSpin$}{GSpin} groups}\label{sec:GSpin-odd}
We consider the split $\GSpin(2n+1)$ over a base field $F$ with $\mathrm{char}(F) \neq 2$, where $n \in \Z_{\geq 1}$. Also put $\GSpin(1) := \Gm$.

Recall from \cite[Proposition 2.4]{As02} that a Borel pair $(B, T)$ over $F$ of $\GSpin(2n+1)$ can be chosen so that $T$ is split and
\begin{align*}
	\mathrm{X}^*(T) & = \bigoplus_{i=0}^n \Z e_i, \\
	Y := \mathrm{X}_*(T) & = \bigoplus_{i=0}^n \Z e_i^*, \\
	\text{simple roots:} & \quad e_1 - e_2, \; \ldots, \; e_{n-1} - e_n, \; e_n, \\
	\text{simple coroots:} & \quad e^*_1 - e^*_2, \; \ldots, \; e^*_{n-1} - e^*_n, \; 2e^*_n - e^*_0,
\end{align*}
where $\{e_i\}_i$, $\{e_i^*\}_i$ are dual bases. Note that the root datum above is dual to that of $\GSp(2n)$.

By \cite[Theorem 2.7]{As02}, Levi subgroups of $\GSpin(2n+1)$ have a form similar to \eqref{eqn:Levi}, namely
\begin{equation}\label{eqn:Levi-GSpin}
	M = \prod_{i=1}^r \GL_{d_i} \times \GSpin(2n^\flat + 1), \quad \sum_{i=1}^r d_i + n^{\flat} = n.
\end{equation}
This can be seen via duality as in \textit{loc.\ cit.}, or more directly by removing the $d_1$, $d_1 + d_2$, \ldots, $d_1 + \cdots + d_r$-th simple roots from the list above, as well as the corresponding simple coroots.

When dealing with quadratic forms on $Y$, we will write the variable corresponding to $e_i$ as $\mathsf{X}_i$, for $i = 0, \ldots, n$.

\begin{lemma}\label{prop:GSpin-prep}
	Set $G = \GSpin(2n+1)$ and take $T \subset G$ as above. Let $Q: Y \to \Z$ be a $\mathrm{W}(G, T)$-invariant quadratic form.
	\begin{enumerate}[(i)]
		\item If $n > 1$, we have
		\[ Q = 2c \left( \sum_{j=1}^n \mathsf{X}_0 \mathsf{X}_j + \mathsf{X}_0^2 \right) + c \sum_{1 \leq j < k \leq n} \mathsf{X}_j \mathsf{X}_k + d \sum_{j=1}^n \mathsf{X}_j^2 \]
		for unique $c, d \in \Z$.
		\item If $n = 1$, we have
		\[ Q = a\mathsf{X}_0^2 + a\mathsf{X}_0 \mathsf{X}_1 + d\mathsf{X}_1^2 \]
		for unique $a, d \in \Z$.
	\end{enumerate}
\end{lemma}
\begin{proof}
	Reflections with respect to roots $e_1 - e_2$, \ldots, $e_{n-1} - e_n$ generate all permutations of $\mathsf{X}_1, \ldots, \mathsf{X}_n$, and keep $\mathsf{X}_0$ fixed. Hence an invariant quadratic form $Q$ on $Y$ can be written as
	\[ Q = a \sum_{j=1}^n \mathsf{X}_0 \mathsf{X}_j + b \mathsf{X}_0^2 + c \sum_{1 \leq j < k \leq n} \mathsf{X}_j \mathsf{X}_k + d \sum_{j=1}^n \mathsf{X}_j^2. \]
	
	Consider the reflection $s_n$ with respect to $e_n$. The corresponding coroot being $2e_n^* - e^*_0$, we see that
	\[ s_n \mathsf{X}_j = \begin{cases}
		\mathsf{X}_0 + \mathsf{X}_n, & \text{if}\; j = 0 \\
		\mathsf{X}_j, & \text{if}\; 1 \leq j < n \\
		-\mathsf{X}_n, & \text{if}\; j = n.
	\end{cases}\]
	
	Therefore,
	\begin{multline*}
		s_n Q = a \sum_{1 \leq j < n} (\mathsf{X}_0 + \mathsf{X}_n)\mathsf{X}_j - a (\mathsf{X}_0 + \mathsf{X}_n)\mathsf{X}_n + b (\mathsf{X}_0 + \mathsf{X}_n)^2 \\
		+ c \sum_{1 \leq j < k \leq n} \mathsf{X}_j \mathsf{X}_k - 2c \sum_{1 \leq j < n} \mathsf{X}_j \mathsf{X}_n + d \sum_{j=1}^n \mathsf{X}_j^2 \\
		= Q + a \sum_{1 \leq j < n} \mathsf{X}_j \mathsf{X}_n -2a \mathsf{X}_0 \mathsf{X}_n - a \mathsf{X}_n^2 + 2b \mathsf{X}_0 \mathsf{X}_n + b \mathsf{X}_n^2 - 2c \sum_{1 \leq j < n} \mathsf{X}_j \mathsf{X}_n .
	\end{multline*}
	When $n > 1$, we see $Q$ is $\mathrm{W}(G, T)$-invariant if and only if $a = b = 2c$.
	
	When $n=1$, the term $\sum_{1 \leq j < n} \mathsf{X}_j \mathsf{X}_n$ disappears, thus $a = b$ and $Q$ also has the desired form.
\end{proof}

\begin{proposition}\label{prop:GSpin}
	Let $G = \GSpin(2n+1)$ with $n > 1$ (resp.\ $n = 1$), and $M \subsetneq G$ as in \eqref{eqn:Levi-GSpin}. Given $\tilde{\mathcal{G}} \in \cate{CExt}(G, \mathbf{K}_2)$ with attached datum $(Q, \mathcal{D}, \varphi) \in \cate{BD}(G, T)$, the pull-back $\tilde{\mathcal{M}}$ of $\tilde{\mathcal{G}}$ to $M$ is the external Baer sum of pull-backs to factors in \eqref{eqn:Levi-GSpin} if and only if the parameter $c$ (resp.\ $a$) in Lemma \ref{prop:GSpin-prep} is zero.
\end{proposition}
\begin{proof}
	This follows by combining Proposition \ref{prop:Baer-orthogonal} and Lemma \ref{prop:GSpin-prep}; see also the discussions after \eqref{eqn:Levi-GSpin}. Specifically, put $s_0 = 0$ and $s_i = d_1 + \cdots + d_i$ for all $1 \leq i \leq r$. Decompose $Y$ into the direct sum of
	\begin{align*}
		Y_i & = \bigoplus_{j=s_{i-1} + 1}^{s_i} \Z e_j^* \quad (1 \leq i \leq r), \\
		Y^\flat & = \Z e_0^*\oplus\bigoplus_{j=s_r + 1}^n \Z e_j^*.
	\end{align*}
	Suppose $n > 1$. Then all crossed terms vanish when $c=0$ by Lemma \ref{prop:GSpin-prep}; for the converse, $M \subsetneq G$ implies $r \geq 1$, and $B_Q(e_0^*,e_1^*) = 2c$. If $n = 1$, then $M \neq G$ forces $r = d_1 = 1$, so $Y_1 = \Z e_1^*$, $Y^\flat = \Z e_0^*$ and $B_Q(e_0^*,e_1^*)=a$. Hence Proposition \ref{prop:Baer-orthogonal} yields the desired criterion.
\end{proof}

\begin{remark}\label{rem:non-split-GSpin}
	For non-split $\GSpin(2n+1)$ with $n > 1$, we still have that $c = 0$ implies the decomposability of $\tilde{\mathcal{M}}$ into an external Baer sum.
\end{remark}

\begin{remark}
	Suppose $F$ is local. Choose $m \mid m_F$ and take the $m$-fold BD-cover $\tilde{G}$ (resp.\ $\tilde{M}$) arising from $\tilde{\mathcal{G}}$ (resp.\ $\tilde{\mathcal{M}}$). One can determine whether $\tilde{M}$ is decomposed or not (cf.\ Definition \ref{def:decomposed-Levi}) by inspecting the expression of $Q$ in Lemma \ref{prop:GSpin-prep}, as explained below.
	\begin{itemize}
		\item In view of Proposition \ref{prop:GSpin}, the condition $c = 0$ implies $\tilde{M}$ is decomposed, by working on the level of $\mathbf{K}_2$-extensions. This also holds for non-split $\GSpin(2n+1)$ by Remark \ref{rem:non-split-GSpin}.
		
		\item Assume $G$ is split. To obtain sharper results, observe that the $F$-points of factors in \eqref{eqn:Levi-GSpin} commute in $\tilde{M}$ if and only if their intersections with $T(F)$ commute. Indeed, this follows from the Bruhat decomposition; see the proof of Proposition \ref{prop:O}. The relevant commutators are then described in terms of Hilbert symbols, cf.\ the discussions in \S\ref{sec:counterexample-unitary} and \S\ref{sec:counterexample-SO4}, or \cite[Corollary 3.14]{BD01}.
	\end{itemize}
\end{remark}

\begin{remark}
	According to \cite[\S 2]{As02}, Levi subgroups of $\GSpin(2n)$ or $\GSp(2n)$ have the same description as \eqref{eqn:Levi-GSpin}, and the invariant quadratic forms $Q: Y \to \Z$ can be determined by a computation akin to Lemma \ref{prop:GSpin-prep}. These cases are omitted since the method of \cite{HW25} does not apply to them thus far.
\end{remark}


\printbibliography[heading=bibintoc]

\vspace{1em}
\begin{flushleft} \small
	W.-W. Li: Beijing International Center for Mathematical Research / School of Mathematical Sciences, Peking University. No.\ 5 Yiheyuan Road, 100871 Beijing, P.~R.~China. \\
	E-mail address: \href{mailto:wwli@bicmr.pku.edu.cn}{\texttt{wwli@bicmr.pku.edu.cn}}
\end{flushleft}

\end{document}